\documentclass[leqno,11pt,draft]{amsart}
\usepackage{amssymb}
\usepackage{amsmath}
\usepackage{enumerate}
\usepackage{amsfonts}

\newtheorem{teo}[equation]{Theorem}

\newtheorem{lema}[equation]{Lemma}
\newtheorem{propo}[equation]{Proposition}
\newtheorem{remark}[equation]{Remark}

\numberwithin{equation}{section}

\begin{document}

\title[Multipliers for Hardy spaces ]
{Remarks on spectral multiplier theorems\\
 on Hardy spaces\\
associated with semigroups of operators}


\begin{abstract} Let $L$ be a non-negative, self-adjoint operator on $L^2(\Omega)$,
where $(\Omega, d, \mu)$ is a space of homogeneous type. Assume
that  the semigroup $\{T_t\}_{t>0}$ generated by $-L$ satisfies
Gaussian bounds, or more generally Davies-Gaffney estimates. We
say that $f$ belongs to the Hardy space $H^1_L$ if the square
function
$$S_h f(x)=\left(\iint_{\Gamma (x)}|t^2Le^{-t^2L}f(y)|^2
\frac{d\mu(y)}{\mu (B_d(x,t))}\frac{dt}{t}\right)^{1\slash 2}$$
 belongs to
$L^1(\Omega, \, d\mu)$, where $\Gamma (x)=\{(y,t)\in \Omega \times
(0,\infty): \ d(x,y)<t\}$. We prove  spectral multiplier theorems
for $L$ on $H^1_L$.

\end{abstract}

\author[Jacek Dziuba\'nski ]{ Jacek Dziuba\'nski }
\address{Instytut Matematyczny\\
Uniwersytet Wroc\l awski\\
50-384 Wroc\l aw \\
Pl. Grunwal\-dzki 2/4\\
Poland} \email{jdziuban@math.uni.wroc.pl,
preisner@math.uni.wroc.pl}

\author[Marcin Preisner ]{ Marcin Preisner}

\thanks{
 Supported by the European Commission Marie Curie Host Fellowship
for the Transfer of Knowledge "Harmonic Analysis, Nonlinear
Analysis and Probability" MTKD-CT-2004-013389 and by Polish
Government funds for science - grant N N201 397137, MNiSW}

\maketitle

\section{Introduction.}

 A classical H\"ormander multiplier theorem \cite{Horm} asserts that
if $m$ is a bounded function on $\mathbb R^d$ such that for some
$\beta >d\slash 2$ and any
 radial function $\eta\in C_c^\infty$, $\text{supp}\, \eta \subset
 \{\xi \in \mathbb R^d: \ 2^{-1}\leq |\xi |\leq 2\}$, one has
 $$\sup_{t>0} \| \eta (\, \cdot \, )m(t\, \cdot \, )\|_{W^{2,\beta}(\mathbb R^d)}
 \leq C_\eta,$$
 where $\| \ \cdot \ \|_{W^{2,\beta}(\mathbb R^d)}$ is the standard Sobolev
 norm on $\mathbb R^d$, then the multiplier operator
 $f\mapsto \mathcal F^{-1}(m\mathcal F
 f)$, initially defined on $L^p(\mathbb R^d)\cap L^2(\mathbb R^d)$,
 is bounded on $L^p(\mathbb R^d)$ for  $1<p<\infty$, and is of weak-type
 (1,1). Here $\mathcal F$ denotes the Fourier transform.

Let $(\Omega, \, d(x,y))$ be a metric space equipped with a
positive measure $\mu$. We assume that $(\Omega, \, d, \, \mu)$ is
a space of homogeneous type in the sense of Coifman-Weiss
\cite{CW}, that is, there exists a constant $C>0$ such that
 \begin{equation}\label{doubling}
 \mu(B_d(x,2t))\leq C\mu (B_d(x,t)) \ \ \
 \text{for every} \ x\in\Omega, \ t>0,
 \end{equation}
 where $B_d(x, t)=\{ y\in\Omega: \ d(x,y)<t\}$.
  The condition (\ref{doubling}) implies that there exist  constants
   $C>0$ and $q>0$ such that
  \begin{equation}\label{growth}
 \mu(B_d(x,st))\leq C_0 s^q \mu(B_d(x,t)) \ \ \ \text{for every }
 x\in \Omega, \ t>0, \ s>1.
 \end{equation}
 Of course we wish  to get $q$ as small as
 possible even at the expense of large $C_0$.

   Let   $\{T_t\}_{t>0}$ be a semigroup of linear operators on $L^2(\Omega, \, d\mu)$
   generated by  $-L$, where $L$ is a non-negative, self-adjoint
    operator which is injective on its domain.
  Assume  the operators $T_t$ have
     the  following form
   \begin{equation}\label{sem}
 T_tf(x)=\int_\Omega T_t(x,y)f(y)d\mu (y),
   \end{equation}
   where  the kernels $T_t(x,y)$ satisfy  Gaussian bounds,
   that is, there exist constants $C_0$, $c_0>0$ such that for
   every $x,y\in\Omega$, $t>0$, we have
   \begin{equation}\label{G1}
   |T_t(x,y)|\leq \frac{C_0}{V(x,\sqrt{t})}
   \exp\left(-\frac{d(x,y)^2}{c_0t}\right),
   \end{equation}
 where  here and subsequently $V(x,t)=\mu (B_d(x,t))$.
 The estimate (\ref{G1}) implies that for every $k\in\mathbb N$
 there exist  constants  $C_k,c_k>0$ such that
 \begin{equation}\label{G2}
   \left|\frac{\partial^k}{\partial t^k}T_t(x,y)\right|
   \leq \frac{C_k}{t^k V(x,\sqrt{t})}
   \exp\left(-\frac{d(x,y)^2}{c_k t}\right) \ \ \text{  for } \
 x,y\in\Omega, \ t>0.
   \end{equation}
   The constants $C_k, \ c_k$  in (\ref{G2}) depend only on $k$ and the constants
   $C, C_0, q,c$ in (\ref{growth}) and (\ref{G1}).

 For a suitable function $f$ (e.g., from $L^2(\Omega)$) we consider the
 square  function $S_hf$ associated with $L$  defined by
\begin{equation}\label{square}
 S_hf(x)=\left(\iint_{\Gamma (x)}|t^2L T_{t^2}
 f(y)|^2\frac{d\mu(y)}{V(x,t)}\frac{dt}{t}\right)^{1\slash 2},
 \end{equation}
where $\Gamma (x)=\{(y,t)\in\Omega \times (0,\infty):\ d(x,y)\leq
t\}$.

Following \cite{ADM},  \cite{AMR},  \cite{Hof} (see also
\cite{BZ}, \cite{DZ00}) we define the Hardy space $H^1_L=H^1_{L,
S_h}(\Omega)$ as the completion of $\{f\in L^2(\Omega):\ \|
S_hf\|_{L^1(\Omega)}<\infty\}$ in the norm $\| f\|_{H^1_L}=\|S_h
f\|_{ L^1(\Omega)}$.

 It was proved in Hofmann, Lu, Mitrea, Mitrea, Yan \cite{Hof} that
   the space $H^1_L$, where $-L$ generates a semigroup having Gaussian bounds,
    admits the following atomic
 decomposition.

  Let $M\geq 1$, $M\in \mathbb N$. A function $a$ is a
  $(1,2,M)$-atom for
  $H^1_L$  if there exist a ball
  $B=B_d(y_0,r)=\{ y\in \Omega: \, d(y,y_0)<r\}$ and a function
  $b\in \mathcal D(L^M)$  such that
  \begin{equation}\label{A1} a=L^Mb;\end{equation}
  \begin{equation}\label{A2} \text{supp}\, L^k b \subset B,  \ \
  k=0,1,...,M ; \end{equation}
  \begin{equation}\label{A3} \| (r^2L)^kb\|_{L^2(\Omega)}
  \leq r^{2M}\mu(B)^{-1\slash 2}, \ \ k=0,1,...,M.
  \end{equation}
The atomic norm $\| f\|_{H^1_{L}\text{-\rm atom}}$ is defined by
$$ \| f\|_{H^1_L\text{-\rm atom}}=\inf \sum_j|\lambda_j|,$$
where the infimum is taken over all representation
$f=\sum_j\lambda_j a_j$, where $a_j$ are $(1,2,M)$-atoms for
$H^1_{L}$, $\lambda_j\in \mathbb C$. Theorem 7.1 of \cite{Hof}
asserts that there exists a constant $C>0$ such that
\begin{equation}\label{dz}
 C^{-1} \| f\|_{H^1_L}\leq  \| f\|_{H^1_L\text{-\rm atom}}\leq C\|
 f\|_{H^1_L}.
\end{equation}

 Let
 \begin{equation}\label{res}
 Lf=\int_0^\infty \lambda \, dE_L(\lambda)f
 \end{equation}
 be the spectral resolution of $L$.

 Our first  goal in  this paper is to present a simple proof of the following spectral
 multiplier theorem.

 \begin{teo}\label{maintheorem}
 Let $m$ be  a bounded function defined on $(0,\infty)$ such
that for  some real number $\alpha > q\slash 2$ and any nonzero
function $\eta\in C_c^\infty (2^{-1}, 2)$ there exists a constant
$C_\eta$ such that
\begin{equation}\label{condition}
\sup_{t>0} \| \eta (\, \cdot \,) m(t\,\cdot \,)\|_{W^{\infty
,\alpha}(\mathbb R)}\leq C_\eta,
\end{equation}
where $\| F \|_{W^{p,\alpha}(\mathbb R)}=\| (I-d^2\slash
dx^2)^{\alpha \slash 2}F\|_{L^p(\mathbb R)}$.
  Then the spectral multiplier operator
\begin{equation}\label{mult_op}
 m(L) = \int_0^\infty m(\lambda) dE_L(\lambda ),
 \end{equation}
 maps $(1,2,1)$-atoms for $H^1_L$ into $H^1_L$. Moreover, there exists a constant $C>0$ such that
 \begin{equation}\label{eqqq1}
 \| m(L)a\|_{H^1_L}\leq C \ \  \ \text{for every} \
 (1,2,1)\text{-atom}.
 \end{equation}
\end{teo}

\begin{remark} {\rm If we additionally assume that  for every
$y\in\Omega$ there exist  constants $\kappa>0$ and $c>0$ such that
$\mu(B_d(y,s))\geq cs^\kappa$ for $s>1$, then the operator $m(L)$
extends uniquely to a bounded operator on $H^1_L$ (see Section 5
for details).}
\end{remark}

\begin{remark}
 {\rm It turns out that if we replace (\ref{condition}) by the stronger condition
 \begin{equation}\label{condition_st}
\sup_{t>0} \| \eta (\, \cdot \,) m(t\,\cdot \,)\|_{W^{2
,\alpha}(\mathbb R)}\leq C_\eta,
 \end{equation}
 with some $\alpha > (q+1)\slash 2$, then the multiplier
 theorem  holds for Hardy spaces
 associated with more general semigroups, that is, semigroups
 satisfying Davies-Gaffney estimates. This will be discussed in
 Section 4. We  present two seemingly similar theorems with two different proofs.
 The first proof, thanks to \cite{Hebisch}, could be also adapted
 to cover the case  of a broader class of  semigroups with integral
 kernels of a very  mild decay. The other one does not even
 require the existence of integral kernels of the semigroups under
 consideration, however depends very much on the finite speed propagation of the wave
 equation associated with generators which is in fact  equivalent  to Davies-Gaffney
 estimates, see, e.g., \cite{S2}, \cite{CS}}.
\end{remark}

 Spectral multiplier theorems on various spaces attracted attention
 of many authors (see, for example, \cite{Alex}, \cite{Bl}, \cite{Christ}, \cite{Cowling},
 \cite{DMM}, \cite{DOS}, \cite{Dz2}, \cite{DZ4}, \cite{FS}, \cite{GMST}, \cite{hebisch0}, \cite{HZ}, \cite{MM}, \cite{MS},
 \cite{S},
 \cite{Stein}, and references there). E. M. Stein \cite{Stein}
 proved that if $-A$ is the infinitesimal generator of a symmetric
 diffusion semigroup and $m$ is of Laplace transform type, then
 $m(A)$ is bounded on $L^p$, $1<p<\infty$.  E. Stein and A.
 Hulanicki (see \cite{FS}) noticed  that if $-A$ is a sublaplacian on a
 stratified Lie group $G$, then the convolution kernel of the operator
 $m(A)$ satisfies Calder\'on-Zygmund type estimates. This fact
 together with atomic decompositions of  the Hardy spaces $H^p(G)$
 leads to a spectral multiplier theorem on these spaces (see \cite[Theorem 6.25]{FS}).
 The finite speed propagation of the wave equation
  was used by Sikora
  \cite {S} and \cite{S2} for proving  $L^p$ bounds for certain
 operators.  Actually, the technique  of the proof of Lemma \ref{dd1} is taken
from \cite{S}.

 The development of the theory of real Hardy spaces in $\mathbb
 R^d$ had its origin in works
 Stein-Weiss \cite{SW} and Fefferman-Stein \cite{FeS}. An
 important contribution to the theory were  atomic decompositions
 proved by Coifman \cite{C} for $d=1$ and Latter \cite{L} for
 $d>1$. The extension of $H^p$ on  the spaces of homogenous type
 is due to Coifman-Weiss \cite{CW} (see also \cite{MS}).
Hardy spaces associated with various semigroups of linear operators
 were considered by many authors. For their properties and equivalent characterizations
 we refer the reader to \cite{ADM}-\cite{BDT},
 \cite{DD}-\cite{DGMTZ}, \cite{GC}, \cite{Hof}, \cite{MM2}, \cite{MPR}.

\section{Functional calculi}

 For $\beta\geq 0$ let $\omega_\beta(x,y)=(1+d(x,y))^\beta$. The function
 is submultiplicative, that is, $\omega_\beta (x,y)\leq
 \omega_\beta (x,z)\omega_\beta (z,y)$.

 For an integral kernel $k(x,y)$ and $\beta >0$ we define
 \begin{equation}\begin{split}\nonumber
 \| k(x,y)\|_{\omega (\beta)} & =\sup_{x\in  \Omega}\int
 |k(x,y)|(1+d(x,y))^\beta d\mu(y) \\
 & \ \ +\sup_{y\in\Omega}\int
 |k(x,y)|(1+d(x,y))^\beta d\mu(x).
 \end{split}\end{equation}

  The following theorem is a consequence of   (\ref{G1}) and results of W. Hebisch
  \cite[Theorem 2.10]{Hebisch}.

  \begin{teo}\label{hebisch}
 Let $(\Omega, d, \, \mu)$ and $\{T_t\}_{t>0}$ satisfy
 (\ref{growth}) and (\ref{G1}) respectively.
 For $\alpha,\ \beta >0$ with $\alpha>\beta+q\slash 2$
 there exists a constant $C'>0$ such
 that for every function $\eta\in W^{\infty, \alpha}(\mathbb R)$ with $\text{\rm supp}\, \eta
 \subset(1\slash 4, 4)$ the multiplier operator
 $$ \eta (L)f=\int_0^\infty \eta (\lambda )\, dE_L(\lambda)f$$
 is of the form
 $$ \eta (L)f(x) =\int_{\Omega } \eta (L)(x,y)f(y)\, d\mu (y)$$
 with
 \begin{equation}\label{wal}
   \| \eta (L)(x,y)\|_{\omega (\beta )} \leq C'\|
 \eta\|_{W^{\infty ,\alpha }(\mathbb R)}.
 \end{equation}
 The constant $C'$ in (\ref{wal}) depends only on $\alpha,\,
 \beta$ and the constants $C, \, q$ from (\ref{growth}) and constants $C_0, \,
 c_0$ from (\ref{G1}).
  \end{teo}

 In this paper we shall use the following  scaling argument. For
 $\tau>0$, let $d^{\{\tau\}}(x,y)=\tau^{-1\slash 2} d(x,y)$. Then the space
 $(\Omega, d^{\{\tau\}}(x,y),\mu)$ is the space of homogeneous type such that
 \begin{equation}\label{growth1}
 V_\tau(x,st)=\mu(B_{d^{\{\tau\}}}(x,st))\leq Cs^q\mu(B_{d^{\{\tau\}}}(x,t)), \ \
 s>1,
 \end{equation}
 with the same constants $C,q$ as in (\ref{growth}).
  Similarly, let
 $L^{\{\tau\}}=\tau L$ and  $\{T_t^{\{\tau\}}\}_{t>0}$ be the
 semigroup generated by $-L^{\{\tau\}}$. Clearly,
 $T^{\{\tau\}}_t(x,y)=T_{\tau t}(x,y)$ are the integral kernels of $T_t^{\{\tau\}}$.
 Hence, for $k=0,1,2,...$, we
 have
 \begin{equation}\label{G4}
   \left|\frac{\partial^k}{\partial t^k}T^{\{\tau\}}_t(x,y)\right|
   \leq \frac{C_k}{t^k V_\tau (x,\sqrt{t})}
   \exp\left(-\frac{d^{\{\tau\}}(x,y)^2}{c_k t}\right) \ \ \text{  for } \
 x,y\in\Omega, \ t>0,
 \end{equation}
 with the same constants $C_k, c_k$ as in (\ref{G1}) and
 (\ref{G2}) independent of $\tau$.

 Therefore, from Theorem \ref{hebisch} we conclude that
 \begin{equation}\label{scale2}\begin{split}
 & \int_{\Omega} |\eta(\tau
 L)(x,y)|\left(1+\frac{d(x,y)}{\sqrt{\tau}}\right)^\beta \,
 d\mu (y)+\int_{\Omega} |\eta(\tau
 L)(x,y)|\left(1+\frac{d(x,y)}{\sqrt{\tau}}\right)^\beta \, d\mu (x)\\
 &\leq C\|
 \eta\|_{W^{\infty ,\alpha}(\mathbb R),}
 \end{split}\end{equation}
 provided  $\text{supp}\, \eta \subset (4^{-1}, 4)$,  $\alpha>\beta+q\slash 2$.

 \begin{propo}\label{Prop1} Assume that $m$ satisfies the assumptions of
 Theorem \ref{maintheorem}. For $N=1,2$, we set
 $$\Phi_t^{\langle N\rangle }(\lambda)=(t^2\lambda)^N
 e^{-t^2\lambda}m(\lambda).$$
Then there exist $\beta>0$ and $C''>0$ such that
\begin{equation}\label{scale_m}\begin{split}
\int_\Omega |\Phi_t^{\langle N\rangle
}(L)(x,y)|\left(1+\frac{d(x,y)}{t}\right)^\beta\, d\mu(x)\leq C'',
\end{split}\end{equation}
\begin{equation}\label{scale_m2}\begin{split}
\int_\Omega |\Phi^{\langle N\rangle}
_t(L)(x,y)|\left(1+\frac{d(x,y)}{t}\right)^\beta d\mu(y) \leq C''.
\end{split}\end{equation}
 \end{propo}
 \begin{proof}
 It suffices to prove (\ref{scale_m}) for $t=1$ and then use the
 scaling argument. Fix a $C_c^\infty (\frac{1}{2},2)$ function
$\psi $ such that
 \begin{equation}\label{ppsi} \sum_{j\in\mathbb Z}
\psi(2^{-j}\lambda )=1 \ \ \ \text{for} \ \lambda >0.
\end{equation}

Denote $n_j(\lambda)=\psi (2^{-j}\lambda)\lambda^N
e^{-\lambda}m(\lambda)$, $\tilde n_j(\lambda)=n_j(
2^{j}\lambda)=\psi (\lambda)(2^j\lambda)^N e^{-2^j\lambda}
m(2^j\lambda)$.
 Clearly, $\text{supp}\, \tilde n_j\subset (2^{-1},
2)$ and
\begin{equation}\label{est_norm}
\| \tilde n_j\|_{W^{\infty ,\alpha}(\mathbb R)}\leq \begin{cases}
C 2^{-j} \ \
& \text{for} \ j\geq 0;\\
C2^{jN}\ \ \ & \text{for } \ j <0.
\end{cases}
\end{equation}
Let $0<\beta \leq 1\slash 2$ be such that
 $\alpha >\beta +q \slash 2$. Applying (\ref{scale2}) combined with
(\ref{est_norm}) we obtain
\begin{equation}\label{NJ}
\int_\Omega |n_j(L)(x,y)|\left(1+2^{j\slash
2}d(x,y)\right)^{\beta}d\mu(x)\leq \begin{cases}C2^{-j} \ \
&\text{for} \ j\geq 0;\\
C2^{jN} \ \ &\text{for } j<0,\\
\end{cases}\end{equation}
 with the same bounds  when integrating with respect to
 $d\mu(y)$. Obviously,
 \begin{equation}\label{s1}
 \int_\Omega |\Phi_1^{\langle N\rangle }(L)(x,y)| \, d\mu(x)\leq \sum_{j\in\mathbb Z}
 \int_\Omega |n_j(L)(x,y)|\, d\mu(x)\leq C'.
 \end{equation}
 Moreover, from (\ref{NJ}) we also deduce  that
\begin{equation}\begin{split}
 \int_\Omega |\Phi_1^{\langle N\rangle }(L)(x,y)|d(x,y)^\beta\, d
 \mu(x) &\leq  \sum_{j\in\mathbb Z} \int_\Omega |n_j(L)(x,y)|d(x,y)^\beta\, d
 \mu(x)\\
 &\leq \sum_{j\geq 0} C 2^{-j-\beta\slash 2}+\sum_{j<0}C'2^{jN-j\beta\slash 2}\leq C',
 \end{split} \end{equation}
 which implies  (\ref{scale_m}) for
 $t=1$. To prove  (\ref{scale_m2}) we proceed in the same way.
 \end{proof}

 \begin{lema}\label{mainlemma} For $N=1$ or $N=2$, let
 \begin{equation}\label{theta}
 \Theta_j^{\langle N\rangle }(x,y)=\sup_{2^j\leq t
 <2^{j+1}}\sup_{d(x,x')<t}|\Phi_t^{\langle N\rangle }(x',y)|.
 \end{equation}
 Then there exist constants $C'>0$ and $\beta>0$ such that
 \begin{equation}
 \int_{\Omega }\Theta_j^{\langle N\rangle }(x,y)\left(1+\frac{d(x,y)}{2^j}\right)^\beta
 d\mu(x)\leq C'.
 \end{equation}
 \end{lema}
 \begin{proof}
 Fix $2^j\leq t<2^{j+1}$ and let $d(x,x')<t$. Since
 $$\Phi_t^{\langle N\rangle }(\lambda)=(2^{1-j}t)^{2N}\exp(-(t^2-2^{2(j-1)}
 )\lambda) \Phi^{\langle N\rangle }_{2^{j-1}}(\lambda)$$
  and $ V(x', (t^2-2^{2(j-1)})^{1\slash
 2})\sim V(x,2^{j})$ for $d(x,x')<t$, we have
 \begin{equation}\label{b1}
 \begin{split}
 |\Phi_t^{\langle N\rangle }(x',y)| &=(2^{1-j}t)^{2N} \left| \int_\Omega
 T_{t^2-2^{2(j-1)}}(x',z)\Phi^{\langle N\rangle }_{2^{j-1}}(z,y)\, d\mu(z)\right|\\
 &\leq C''\int_\Omega
 \frac{\exp(-d(x',z)^2\slash c_0(t^2-2^{2(j-1)}))}{V(x', (t^2-2^{2(j-1)})^{1\slash
 2})} |\Phi^{\langle N\rangle }_{2^{j-1}}(z,y)|\, d\mu(z)\\
 &\leq C\int_\Omega
 \frac{\exp(-d(x,z)^2\slash c'2^{2j})}{V(x, 2^j)} |\Phi^{\langle N\rangle }_{2^{j-1}}(z,y)|\, d\mu(z).\\
 \end{split}\end{equation}
Using (\ref{b1}) and  Proposition \ref{Prop1}  we obtain
\begin{equation}\begin{split}\nonumber
 &\int_\Omega \Theta^{\langle N\rangle }_j(x,y)\left(1+\frac{d(x,y)}{2^j}\right)^\beta\,
 d\mu(x)\\
 & \leq C\int_\Omega\int_\Omega \frac{\exp(-\frac{d(x,z)^2}{ c'2^{2j}})}{V(x, 2^j)}
 |\Phi^{\langle N\rangle }_{2^{j-1}}(z,y)|\,\left(1+\frac{d(x,z)}{2^j}\right)^\beta
 \left(1+\frac{d(z,y)}{2^j}\right)^\beta   d\mu(z)\, d\mu(x)\\
 & \leq C'. \end{split}\end{equation}
 \end{proof}

\section{Proof of Theorem \ref{maintheorem}}

 It suffices
 to establish that
 there exists a constant $C$ such that for every $(1,2,1)$-atom $a$ for $H^1_L$
 we have
 \begin{equation}\label{reduction444}
\|S_h( m(L) a)\|_{L^1(\Omega)}\leq C.
 \end{equation}

 Our proof of (\ref{reduction444}) borrows ideas from \cite{Dz9}.
 Let $a$ be a $(1,2,1)$-atom for $H^1_L$ and let $b$ and $B=B_d(y_0,r)$ be as
in
 (\ref{A1})--(\ref{A3}).
 Since $S_h$ is bounded on $L^2(\Omega)$, we have
  \begin{equation}\label{C1}
  \| S_h m(L)a\|_{L^1(B_d(y_0, 2r), d\mu)} \leq C' \| m(L)a\|_{L^2(\Omega)}
  \mu(B)^{1\slash 2}\leq C\| a\|_{L^2(\Omega)}\mu(B)^{1\slash 2} \leq C.
  \end{equation}
  It suffices to estimate $S_h m(L)a$ on $(2B)^c$, where $2B=B_d(y_0,2r)$.
  Clearly, $\Phi^{\langle 1\rangle}_t(L)a =t^{-2}\Phi_t^{\langle
  2\rangle}(L)b$.
  Set $j_0=\log_2 r$.  Then
  \begin{equation}\begin{split}\nonumber
 (S_h m(L) & a(x))^2  = \iint_{\Gamma (x)}
 \left|t^2(LT_{t^2}m(L)a)(x')\right|^2\frac{d\mu(x')}{V(x',t)}\frac{dt}{t}\\
 &=\sum_{j\in\mathbb Z} \int_{2^j}^{2^{j+1}}\int_{d(x,x')<t}
\left| \Phi_t^{\langle 1\rangle}(L)a(x')\right|^2\frac{d\mu(x')}{V(x',t)}\frac{dt}{t}\\
&=\sum_{j\leq j_0}\int_{2^j}^{2^{j+1}}\int_{d(x,x')<t}
\left| \Phi_t^{\langle 1\rangle}(L)a(x')\right|^2\frac{d\mu(x')}{V(x',t)}\frac{dt}{t}\\
&\ \ + \sum_{j> j_0}\int_{2^j}^{2^{j+1}}\int_{d(x,x')<t} \left|
t^{-2}\Phi_t^{\langle 2\rangle}(L)b(x')\right|^2
\frac{d\mu(x')}{V(x',t)}\frac{dt}{t}.\\
 \end{split}\end{equation}
 Using (\ref{theta}) we have
\begin{equation}\label{c1}\begin{split}
(S_h m(L)  a(x))^2
  &\leq C \sum_{j\leq j_0}\int_{2^j}^{2^{j+1}}\int_{d(x,x')<t}
\left( \int_\Omega \Theta_j^{\langle 1\rangle}(x,y)|a(y)|\,
d\mu(y)
\right)^2\frac{d\mu(x')}{V(x',t)}\frac{dt}{t}\\
& \  +C\sum_{j > j_0}\int_{2^j}^{2^{j+1}}\int_{d(x,x')<t} \left(
\int_\Omega t^{-2}\Theta_j^{\langle
2\rangle}(x,y)|b(y)|d\mu(y)\right)^2
\frac{d\mu(x')}{V(x',t)}\frac{dt}{t}\\
 & \leq C \sum_{j\leq j_0}
\left( \int_\Omega \Theta_j^{\langle 1\rangle}(x,y)|a(y)|\,
d\mu(y)\right)^2\\
&\ \ \ \ + C \sum_{j > j_0} \left( \int_\Omega
2^{-2j}\Theta_j^{\langle 2\rangle}(x,y)|b(y)|\,
d\mu(y)\right)^2, \\
  \end{split}\end{equation}
because
$$\int_{2^j}^{2^{j+1}} \int_{d(x,x')<t}\frac{d\mu (x')}{V(x',t)}\frac{dt}{t}\leq
C.$$

 From (\ref{c1}) we trivially get
 \begin{equation}\begin{split}\nonumber
 (S_h &m(L)a)(x)\\
 &\leq C \left(\sum_{j\leq j_0}
\int_\Omega \Theta_j^{\langle 1\rangle}(x,y)|a(y)|\, d\mu(y)
 +  \sum_{j > j_0}  \int_\Omega 2^{-2j}\Theta_j^{\langle
2\rangle}(x,y)|b(y)|\, d\mu(y)\right).
\end{split}\end{equation}

  Applying Lemma \ref{mainlemma} we obtain

 \begin{equation}\begin{split}\label{dd00}
 & \int_{(2B)^c } (S_h m(L)a)(x)\, d\mu(x)\\
 &\leq C \sum_{j\leq j_0}
 \int_{(2B)^c }  \int_{d(y,y_0)<r} \Theta_j^{\langle 1\rangle}(x,y)
 \left(\frac{d(x,y)}{2^j}\right)^\beta \left(\frac{2^j}{r}\right)^\beta |a(y)|\,
 d\mu(y)\, d\mu(x)\\
 &
 \ \ \ + C \sum_{j > j_0} \int_{(2B)^c} \int_{d(y,y_0)<r} 2^{-2j}\Theta_j^{\langle
2\rangle}(x,y)|b(y)|\, d\mu(y)\, d\mu(x)\\
&\leq C \sum_{j\leq j_0}
  \int_{d(y,y_0)<r}  \left(\frac{2^j}{r}\right)^\beta |a(y)|\,
 d\mu(y)
 + C \sum_{j > j_0} 2^{-2j} \int_{d(y,y_0)<r} |b(y)|\, d\mu(y)\, d\mu(x).\\
\end{split}\end{equation}
 By the Cauchy-Schwarz inequality  $\| a\|_{L^1(\Omega)}\leq 1$ and
 $\|b\|_{L^1(\Omega)}\leq r^2$. Since $2^{j_0}\sim r$, we easily
 conclude  from (\ref{dd00}) that

 \begin{equation}\begin{split}\nonumber
 & \int_{d(x,y_0)>2r } (S_h m(L)a)(x)\, d\mu(x)\leq C,
\end{split}\end{equation}
which together with (\ref{C1}) completes the proof of
(\ref{reduction444})

\section{Spectral multiplier theorem for semigroups satisfying
Davies-Gaffney estimates.}

 Let $\{\mathcal T_t\}_{t>0}$ be a semigroup of linear
  operators on $L^2(\Omega)$ generated by $-\mathcal L$,
  where $\mathcal L$ is a non-negative, self-adjoint operator which is injective on its domain.
   We assume that
   $\{\mathcal T_t\}_{t>0}$  satisfies
 Davies-Gaffney estimates, which briefly  speaking means that
 \begin{equation}\label{D-G} |\langle \mathcal T_t f_1,f_2\rangle |\leq C
 \exp\left(-\frac{\text{dist}(U_1,U_2)^2}{ct}\right)\|
 f_1\|_{L^2(\Omega)}\| f_2\|_{L^2(\Omega)}
 \end{equation}
 for every $f_i\in L^2(\Omega)$, $\text{supp}\, f_i\subset U_i$,
 $i=1,2$, $U_i$ are open subsets of $\Omega$
  (see e.g., \cite{CS}, \cite{Hof} for details).

 The Hardy space
 $H^1_{\mathcal L}$, defined as in Section 1 by means of
 $L^1(\Omega)$ bounds of the square function (\ref{square}), were
 considered by Auscher, McIntosh, Russ \cite{AMR} and Hofmann, Lu,
 Mitrea, Mitrea, Yan \cite{Hof}. It was proved in \cite{Hof} that
 the space $H^1_{\mathcal L}$
  admits atomic decompositions into
 $(1,2,M)$-atoms associated with $\mathcal L$,
 provided $M>q\slash 4$, $M\in \mathbb N$ (see \cite{Hof}).
 Clearly, $L$ is replaced by $\mathcal L$ in the definition
 (\ref{A1})--(\ref{A3}) of $(1,2,M)$-atoms for $H^1_{\mathcal L}$.

 In this section we show that the following spectral multiplier theorem holds
 for Hardy spaces associated with   semigroups satisfying the Davies-Gaffney
 estimates.

 \begin{teo}\label{maintheorem2}
 Let $M>q\slash 4$, $M\in\mathbb N$. Assume  $m$ be  a bounded function defined on $(0,\infty)$ such
that for  some real number $\alpha > (q+1)\slash 2$ and any
nonzero function $\eta\in C_c^\infty (2^{-1}, 2)$ the condition
(\ref{condition_st}) holds.
  Then there exists a constant $C>0$ such that
  \begin{equation}\label{eqqq2}
 \| m(\mathcal L)a\|_{H^1_{\mathcal L}}\leq C \ \ \ \text{for every
 } \ (1,2,2M)\text{-atom}\ a \ \ \text{for the space} \ H^1_{\mathcal L}.
  \end{equation}
\end{teo}

 Fix $\varepsilon >0$ and $M>q\slash 4$, $M\in\mathbb N$.
  We say that a function $\tilde a$ is a
 $(1,2,M,\varepsilon )$-molecule associated to $\mathcal L$ if
 there exist a function $\tilde b\in \mathcal D(\mathcal L^M)$ and
 a ball $B=B_d(y_0,r)$ such that
 \begin{equation}\label{molec1}
  \tilde a=\mathcal L^M \tilde b;
  \end{equation}
 \begin{equation}\label{molec2} \|(r^2\mathcal L)^k\tilde b\|_{L^2(U_jB))}
 \leq r^{2M }2^{-j\varepsilon }V(y_0, 2^jr)^{-1\slash  2}
 \end{equation}
for  $ k=0,1,...,M$,  $j=0,1,2,...$, where
$U_0=B$, $U_j(B)=B_d(y_0,2^jr)\setminus B_d(y_0, 2^{j-1}r)$ for $j\geq 1$.

 It was proved in \cite[Corollary 5.2] {Hof} that every
 $(1,2,M,\varepsilon)$-molecule $\tilde a$ belongs to $H^1_{\mathcal L}$ and
 \begin{equation}\label{molec3}
 \| \tilde a\|_{H^1_{\mathcal L}}\leq C_{\varepsilon,M}.
 \end{equation}

 Of course the condition (\ref{condition_st}) is invariant under the change of
 variable
 $\lambda \mapsto \lambda^s$ in  multipliers. Hence (\ref{eqqq2})
   will be established if we have proved the following
 proposition for $\sqrt{\mathcal L}$.

 \begin{propo}\label{prop_molec}
 Assume that $m$ satisfies (\ref{condition_st}). Fix $M>q\slash 4$,
 $M\in\mathbb N$. Then there exists $\varepsilon >0$ such that for
 every $(1,2,2M)$-atom $a$ for $H^1_{\mathcal L}$ the function
 $$ \tilde a(x) =m(\sqrt{\mathcal L})a(x)$$
 is a multiple of $(1,2,M, \varepsilon)$-molecule. The multiple
 constant is independent of $a$.
 \end{propo}
 \begin{proof}
 Let $a$ be a $(1,2,2M)$-atom for $H^1_{\mathcal L}$
  and let $b$ and $B=B_d(y_0, r)$ be as
 in (\ref{A1})--(\ref{A3}). Set $\tilde b=m(\sqrt{\mathcal
 L})\mathcal L^M b$. Clearly, $\tilde a=\mathcal L^M\tilde b$.
 In order to complete the proof of  the proposition is suffices to verify
 (\ref{molec2}). To do this we need the following lemma.

 \begin{lema}\label{dd1}
 Let $\gamma >1\slash 2$, $\beta >0$. Then there exists a constant $C>0$ such that
 for every even function  $F\in W^{2,\gamma
 +\beta\slash 2}(\mathbb R)$ and every  $g\in
 L^2(\Omega)$, $\text{\rm supp}\, g\subset B_d(y_0, r)$,
 we have
  $$ \int_{d(x,y_0)>2r} |F(2^{-j}\sqrt{\mathcal
  L})g(x)|^2\left(\frac{d(x,y_0)}{r}\right)^\beta d\mu(x)\leq C
  (r2^j)^{-\beta}
  \| F\|_{W^{2,\gamma +\beta\slash 2}}^2\| g\|_{L^2(\Omega)}^2$$
  for $j\in\mathbb Z$.
 \end{lema}
 \begin{proof}[Proof of the lemma]
  The lemma seems to be well-known. For  the convenience of the reader we provide
  a proof. To  this end we borrow methods from  \cite{S}.
  Since $F$ is even,
 $$ F(2^{-j}\sqrt{\mathcal L})g=\frac{1}{2\pi}\int_{\mathbb R} \hat F(\xi)\cos
 (2^{-j}\xi\sqrt{\mathcal L})g\, d\xi,$$
 where $\hat F=\mathcal F F$ is the Fourier transform of $F$.
 The Davies-Gaffney estimates (\ref{D-G}) imply  the finite speed
  propagation of the wave equation $\mathcal L u +u_{tt}=0$ (see, e.g., \cite{S2}, \cite{CS}),
 which means that there exists a constant $C'>0$ that
  $$\text{supp}\, \cos (2^{-j}\xi \sqrt{\mathcal L})g\subset
 B_d(y_0, r+C' 2^{-j}|\xi|).$$
 Hence,
 \begin{equation}\begin{split}\nonumber
 & \left( \int_{d(x,y_0)>2r}  |F(2^{-j}\sqrt{\mathcal
  L})g(x)|^2\left(\frac{d(x,y_0)}{r}\right)^\beta
  d\mu(x)\right)^{1\slash 2}\\
  &=\left( \int_{d(x,y_0)>2r}  \left|\frac{1}{2\pi}
  \int_{\mathbb R}\hat F(\xi )\cos (2^{-j}\xi \sqrt{\mathcal
  L})g(x)\, d\xi \right|^2\left(\frac{d(x,y_0)}{r}\right)^\beta
  d\mu(x)\right)^{1\slash 2}\\
  &\leq C \int_{C2^{-j}|\xi|>2r} \left(\int_{2r<d(x,y_0)<C2^{-j}|\xi|}
  |\hat F(\xi)|^2|\cos (2^{-j}\xi \sqrt{\mathcal L})g(x)|^2 \frac{d(x,y_0)^\beta}{r^\beta}
  d\mu(x) \right)^{1\slash 2}d\xi\\
  &\leq C\int_{\mathbb R} |\hat F(\xi)|
  \left(\frac{2^{-j}|\xi|}{r}\right)^{\beta \slash 2}\| g\|_{L^2(\Omega)} \,
  d\xi\\
  &\leq C''(r2^j)^{-\beta\slash 2}
  \| F\|_{W^{2,\gamma +\beta\slash 2}}\| g\|_{L^2(\Omega )}.
 \end{split}\end{equation}
  \end{proof}
  We are now in a position to complete the proof of Proposition \ref{prop_molec}.
  Fix $\varepsilon >0$ and $\gamma >1\slash 2$ such that
   $\gamma +\varepsilon +q\slash 2=\alpha$.
   Set $\beta =q+2\varepsilon$.
  Then $\gamma +\beta\slash 2=\alpha$. Let $j_0=-\log_2 r$.
  For  an integer number $k$, $0\leq k\leq M$,  write
  \begin{equation}\label{LLL}
  \begin{split} (r^2\mathcal L)^k\tilde b&=r^{2k}\sum_{j\geq j_0}\psi (2^{-j}\sqrt{\mathcal L})
  m(\sqrt{\mathcal L})\mathcal L^{k+M}b +r^{2k}\sum_{j<j_0}\psi (2^{-j}\sqrt{\mathcal L})
  \mathcal L^M m(\sqrt{\mathcal L})\mathcal L^{k}b \\
  &=r^{2k}\sum_{j\geq j_0}\psi (2^{-j}\sqrt{\mathcal L})
  m(\sqrt{\mathcal L})g_1 +r^{2k}\sum_{j<j_0}\psi (2^{-j}\sqrt{\mathcal L})
  \mathcal L^M m(\sqrt{\mathcal L}) g_2,
  \end{split}\end{equation}
  where $g_1=\mathcal L^{k+M}b$, $g_2=\mathcal L^k b$.
 Since $a$ is a $(1,2,2M)$-atom for $\mathcal L$ associated with
 $B=B_d(y_0,r)$ and $b$ (see (\ref{A1})-(\ref{A3})), we have
 \begin{equation}\label{gg}
  \| g_1\|_{L^2(\Omega)}\leq r^{2M-2k}\mu(B)^{-1\slash 2}, \ \ \
  \| g_2\|_{L^2(\Omega)}\leq r^{4M-2k}\mu(B)^{-1\slash 2}.
 \end{equation}
  Put
  \begin{equation}
 F_j(\lambda)=\begin{cases}
 m(2^j\lambda )\psi (\lambda) \ \ & \text{for} \ j\geq
 j_0;\\
2^{2Mj}m(2^j\lambda)\lambda^{2M} \psi
  (\lambda) \ \ &\text{for} \ j<j_0;\\
  \end{cases}
  \end{equation}
 and extend each  $F_j$ to the even function.
 Clearly,
  \begin{equation}\label{FF}
  \|  F_j\|_{W^{2,\alpha}(\mathbb R)}\leq
  \begin{cases}
  C \ \ &\text{for} \ j\geq j_0;\\
 C 2^{2Mj} \ \ &\text{for}  \ j<j_0.
 \end{cases}\end{equation}
  Using Lemma \ref{dd1} combined with (\ref{LLL}) -- (\ref{FF}),  we get
  \begin{equation}\label{weight1}\begin{split}
 & \left(\int_{d(x,y_0)>2r}
  |(r^2\mathcal L)^k \tilde b(x)|^2
 \frac{d(x,y_0)^\beta}{r^\beta} d\mu (x)\right)^{1\slash
 2}\\
 & \ \ \ \ \  \leq Cr^{2k} \sum_{j\geq j_0} (r2^j)^{-\beta\slash 2} \|
 F_j\|_{W^{2,\alpha}(\mathbb R)} \| g_1\|_{L^2(\Omega)} \\
 & \ \ \ \ \ \ \ \ +  Cr^{2k} \sum_{j <j_0} (r2^j)^{-\beta\slash 2} \|
 F_j\|_{W^{2,\alpha}(\mathbb R)} \| g_2\|_{L^2(\Omega)}\\
 &\ \ \ \ \ \ \leq C r^{2M}\mu (B)^{-1\slash
 2}.
  \end{split}\end{equation}
 Moreover,
 \begin{equation}\label{weight2}\begin{split}
  \| (r^2\mathcal L)^k \tilde b\|_{L^2(\Omega)}& =\|
 r^{2k}m(\sqrt{\mathcal L}) \mathcal L^{k+M}b\|_{L^2(\Omega)}\\
 &\leq
 Cr^{2k} \| m\|_{L^\infty (\mathbb R)} \| g_1\|_{L^2(\Omega)}
 \leq C r^{2M}\mu(B)^{-1\slash 2}.
 \end{split}\end{equation}

Let $j\in\mathbb Z$, $j\geq 0$. Applying (\ref{weight1}) and
(\ref{weight2}) we obtain
\begin{equation}\begin{split}
 \| (r^2 \mathcal L)^k \tilde b\|_{L^2(U_j(B))}^2
 &  \leq C \int_{U_j(B)}
|(r^2 \mathcal L)^k \tilde b(x)|^2
\left(1+\frac{d(x,y_0)}{r}\right)^\beta 2^{-j\beta}d\mu(x)\\
&\leq C r^{4M}2^{-j\beta} \mu (B)^{-1}\\
&\leq C'' r^{4M} 2^{-2j\varepsilon} V(y_0, 2^{j}r)^{-1},
 \end{split}\end{equation}
 where in the last inequality we have  used (\ref{growth}).
 \end{proof}

 \section{Remarks}

 {\bf 1.} Assume that $-L$ generates a semigroup with Gaussian bounds.
  If we additionally assume that the space $(\Omega, d, \mu)$
 is such that for every $y\in \Omega$ there exists
 $\kappa=\kappa(y)$ and $c=c(y)>0$ such that
 \begin{equation}\label{lower}
 \mu (B_d(y, s))\geq cs^\kappa \ \ \text{for } s>1,
 \end{equation}
 then the multiplier operator $m(L)$ (see  (\ref{mult_op}))
 extends uniquely to a bounded operator on $H^1_L$. To see this we
 define the space $\mathfrak T$ of test functions in the following
 way: a function $g$ belongs to $\mathfrak T$ if there exist
 $t>0$, a ball $B_d(y,r)$, and   a function $\zeta \in L^\infty(\Omega)$ such
 that   $\text{supp}\, \zeta \subset B_d(y,r)$ and $g=T_t \zeta$.
 Wa say that $g_n$ converge to $g_0$ in $\frak T$ if there exist
 $t>0$, a ball $B=B_d(y,r)$, and functions $\zeta_n$ such
 that $\text{supp}\, \zeta_n \subset B$, $\sup \|
 \zeta_n\|_\infty <\infty$, $g_n=T_t\zeta_n$,  and $\zeta_n(x)\to \zeta_0(x)$ a.e.
 Clearly, $\frak T\subset L^p(\Omega)$ for every $1\leq p\leq
 \infty$. One can easily prove that if $f\in L^1(\Omega)$ is such
 that $\int_\Omega fg\, d\mu=0$ for every $g\in \frak T$, then
 $f=0$.

 \begin{lema}\label{distribution}
 Assume that $m$ satisfies (\ref{condition}). Then $\bar m(L)$ maps
 continuously $\frak T$ into $L^\infty(\Omega)$.
 \end{lema}
 \begin{proof}
 Recall that $\alpha >q\slash 2$. Observe that there exists a constant
 $C>0$ such that for every function $n(\lambda)$ such that $n\in
 W^{\infty, \alpha}(\mathbb R)$, $\text{supp}\, n\subset (2^{j-1},
 2^{j+1})$, one has
 \begin{equation}\label{eeq1}
 |n(L)(x,y)|\leq C
 \mu(B_d(y,2^{-j\slash 2}))^{-1}\| n(2^{j}\, \cdot\,)\|_{W^{\infty, \alpha}(\mathbb R)}.
 \end{equation}
 It suffices to prove (\ref{eeq1}) for $j=0$ and then use the scaling
 argument. Set $\xi(\lambda)=e^\lambda n(\lambda )$. Then $\|
 \xi\|_{W^{\infty, \alpha}(\mathbb R)}\sim\| n\|_{W^{\infty , \alpha}(\mathbb R)}$.
 Hence,
 by Theorem \ref{hebisch},
 $$ |n(L)(x,y)|\leq \int |\xi (L)(x,z)T_1(z,y)|\, d\mu (z)\leq C\mu(B_d(y, 1))^{-1}
 \|
 n\|_{W^{\infty,\alpha}(\mathbb R) }.$$
 Assume that $g\in \frak T$. Then there are $t>0$, $B=B_d(y_0, r)$,
 and a bounded function $\zeta$ such that $g=T_t\zeta$,
 $\text{supp}\, \zeta \subset B$. Of course we can assume that
 $r>1$.
 Let $\psi$ be as in (\ref{ppsi}). Let $n_j(\lambda)=m(\lambda)
 \psi(2^{-j}\lambda) e^{-t\lambda}$. Then
 $$ \bar m(L)g(x)=\sum_j \bar n_j(L)\zeta (x)=\sum_j
  \int \bar n_j(L)(x,y) \zeta (y)\, d\mu(y).$$
 Set $j_0=-2\log_2 r$.  Obviously $\| \bar n_j (2^{j}\, \cdot \, )\|_{W^{\infty,
 \alpha}(\mathbb R)}\leq C e^{-c2^{j}t}$. Thus
 \begin{equation}\label{eeqq2}
 \sum_{j} |\bar n_j(L)(x,y)|\leq C \sum_{j\leq j_0} \mu (B_d(y,2^{-j\slash
 2}))^{-1} +C\sum_{j>j_0} e^{-ct2^j}\mu(B_d(y,2^{-j\slash
 2}))^{-1}.
 \end{equation}
 By (\ref{lower}) there exist  $c(y_0)$ and $\kappa=\kappa(y_0)>0$
 such that for $y\in B_d(y_0,r)$ and $j\leq j_0$ we have
 \begin{equation}\label{eeqq3}
 \mu(B_d(y,2^{-j\slash 2})) \sim \mu( B_d(y_0,2^{-j\slash 2}))
 \geq c(y_0)2^{-j\kappa\slash 2}.
 \end{equation}
 On the other hand, by (\ref{growth}), for $y\in B_d(y_0,r)$ and $j>j_0$, we have
 \begin{equation}\label{eeqq4}
\mu (B_d(y_0, r))\sim \mu (B_d(y, r))\leq C(2^{j\slash
2}r)^q\mu(B_d(y, 2^{-j\slash 2})).
 \end{equation}
 From (\ref{eeqq2})-(\ref{eeqq4}) we conclude that there exists a
 constant $C(y_0,r)$ such that
 $$ \sum_{j} |\bar n_j(L)(x,y)|\leq C(y_0,r) \ \ \ \text{for } \
 x\in \Omega \ \ \text{and}\ y\in B_d(y_0, r).$$
 \end{proof}
 We are now in a position to define the action of $m(L)$ on the
 space $L^1(\Omega)$ in the weak (distributional) sense by putting
 $$\langle m(L)f,g\rangle = \int_\Omega f(x) \overline{\bar m(L)  g(x)}\,
 d\mu (x).$$
  Let us observe  that $m(L)$ is uniquely defined on $H^1_{L}$.
 Indeed, if $f=\sum_j \lambda_j a_j$, where $a_j$ are
 $(1,2,1)$-atoms, $\lambda_j\in\mathbb C$, $\sum
 |\lambda_j|\sim \| f\|_{H^1_{L}}<\infty$ then, by Theorem \ref{maintheorem} and Lemma
 \ref{distribution},  for every $g\in \frak T$ we have
 \begin{equation}\begin{split}\nonumber
 \langle m(L) f,g\rangle & =
 \int_\Omega \Big(\sum_j \lambda_j a_j
 (x)\Big) \overline{\bar m(L) g(x)}\, d\mu(x)\\
 &=\sum_j \lambda_j \int_\Omega a_j
 (x) \overline{\bar m(L) g(x)}\, d\mu(x)\\
 &= \sum_j \lambda_j \int_\Omega
 m(L)a_j (x) \overline{g(x)}\, d\mu (x)\\
&= \int_\Omega \Big( \sum_j \lambda_j
 m(L)a_j (x)\Big)\overline{ g(x)}\, d\mu (x).\\
 \end{split}\end{equation}
Since $\sum_j \lambda_jm(L)a_j$ belongs to $L^1(\Omega)$, we
obtain that $m(L)f=\sum_j \lambda_jm(L)a_j$, which gives the
required uniqueness. Obviously, $\| m(L)f\|_{H^1_L}\leq C\|
f\|_{H^1_L}.$

 \

 {\bf 2.} One of distinguished examples of   semigroups of linear operators
 with
 Gaussian bounds is that generated by a Schr\"odinger operator
  $-A=\Delta -V$ on $\mathbb R^d$,
where $V $ is a nonnegative potential  such that $V\in
L^1_{loc}(\mathbb R^d)$. By  the Feynman-Kac formula the integral
kernels  $p_t(x,y)$  of the semigroup $e^{-tA}$ satisfy
$$ 0\leq p_t(x,y)\leq (4\pi t)^{-d\slash 2}\exp(-|x-y|^2\slash
4t).
$$
 Clearly, considering $(\mathbb R^d, \, d(x,y)=|x-y|, \, dx)$ as a
 space of homogeneous type, we have that (\ref{growth}) and (\ref{lower}) hold with
 $q=d$. Thus, as a corollary of Theorem \ref{maintheorem}, we obtain
 that  any bounded function $m:(0,\infty)\to \mathbb C$ which
 satisfies (\ref{condition}) with $\alpha> d\slash 2$ is an
 $H^1_A$ spectral  multiplier for $A$.

 We would like to remark that the space $H^1_A$ admits also
 characterization by means of maximal function from the semigroup
 $e^{-tA}$
 (see \cite{Hof}). Using arguments similar to those  of \cite{DZ4}
 one can prove the spectral multiplier theorem on Hardy spaces
 associate with the Schr\"odinger operators by applying both atomic
 and maximal function characterizations.

 Another  molecule decomposition of Hardy space $H^1$ associated with
semigroups generated by Schr\"odinger operators was communicated
to us by Jacek Zienkiewicz \cite{Zien}. These decompositions also
lead to  multiplier theorems.

\

{\bf Acknowledgment.} The authors would like to
 thank  Pascal Auscher, Fr\'ed\'eric Bernicot, and  Pawe\l \
G\l owacki for their valuable comments.


\begin{thebibliography}{99}
\baselineskip=6mm

\bibitem{Alex} G. Alexopoulos, {\em Spectral multipliers on Lie groups of polynomial
growth,} Proc. Amer. Math. Soc. 120 (1994), no. 3, 973--979.


\bibitem{ADM} P. Auscher, X.T. Duong, A. McIntosh, {\it Boundedness
of Banach space valued singular integral operators and Hardy
spaces}, preprint.

\bibitem{AMR} P. Auscher, A. McIntosh, E. Russ, {\it
 Hardy Spaces of Differential Forms on Riemannian Manifolds}, J. Geom. Anal.
 18 (2008), 192--248.


\bibitem{BZ}  F. Bernicot, J. Zhao, {\it New abstract Hardy spaces},
J. Funct. Anal. 255 (2008), 1761--1798.


 \bibitem{BDT} J. Betancor, J. Dziuba\'nski, and J. L. Torrea,
{\it On Hardy spaces associated with Bessel operators}, J. Anal.
Math. 107 (2009), 195--219.

\bibitem{Bl} S. Blunck, {\em A H\"ormander-type spectral multiplier theorem for
operators without heat kernel}, Ann. Sc. Norm. Super. Pisa Cl.
Sci. (5), 2 (2003)  no. 3,  449-–459.


\bibitem{Christ} M. Christ, {\em $L^p$ bounds for spectral
multipliers on nilpotent groups,} Trans Amer. Math. Soc. 328 (1991),
73--81.

\bibitem{C} R. Coifman, {\em A real variable characterisation of
$H^p$}, Studia Math., 51 (1974), 269--274.

\bibitem{CW} R. Coifman and G. Weiss, {\it Extensions of Hardy spaces
and their use in analysis}, Bull. Amer. Math. Soc. 83 (1977),
569--645.


\bibitem{CS} T. Coulhon and A. Sikora, {\em Gaussian heat kernel
upper bounds via Phragm\'en-Lindel\"of theorem},  Proc. Lond.
Math. Soc. (3) 96 (2008), no. 2, 507--544.

\bibitem{Cowling} M. Cowling, {\em Harmonic analysis on semigroups,} Ann.
of Math. 117 (1983), 267--283.

\bibitem{DMM} L. De Michele and G. Mauceri, {\em $H^p$ multipliers on
stratified groups,} Ann. Mat. Pura Appl. 148 (1987), 353--366.

\bibitem{DD} D. Deng, X. Duong, A. Sikora, L. Yan, {\it
Comparison of the classical BMO with the BMO spaces associated with
operators and applications,}  Rev. Mat. Iberoam. 24 (2008), no. 1,
267--296.


\bibitem{DY} X. Duong, L. Yan, {\it Duality of Hardy and BMO spaces associated with
operators with heat kernel bounds,} J. Amer. Math. Soc. 18, Number
4 (2005), 943--973.


\bibitem{DOS} X. T. Duong, E. M. Ouhabaz, and A. Sikora, {\em Plancherel-type
estimates and sharp spectral multipliers,} J. Funct. Anal., 196(2)
(2002), 443-–485.


\bibitem{D0} J. Dziuba\'nski, {\it Atomic decomposition of $H^p$
spaces associated with some Schr\"odinger operators}, Indiana Univ.
Math. J. 47 (1998), 75--98.

\bibitem{Dz9} J. Dziuba\'nski,  {\it Spectral multipliers for $H^1$ spaces
associated with some Schr\"odinger operators,} Proc.  Amer. Math.
Soc. 127 (1999), 3605--3613.

\bibitem{Dz2} J. Dziuba\'nski,  {\it Spectral multipliers for Hardy spaces
associated with Schr\"odinger operators with polynomial potentials,}
Bull. London Math. Soc. 32 (2000), no. 5, 571--581.

\bibitem{DZ4} J. Dziuba\'nski, {\em A spectral multiplier theorem
for $H^1$ spaces associated with Schr\"odinger operators with
potentials satisfying a reverse H\"older inequality,}, Illinois J.
Math. 45 (2001), 1301-1313.


\bibitem{D1}  J. Dziuba\'nski, {\it Note on $H^1$ spaces related to
degenerate Schr\"odinger operators}, Ill. J.  Math.  49.4, (2005),
1271--1297.

\bibitem{D8} J. Dziuba\'nski, {\it Hardy spaces for Laguerre expansions},
Constructive Approximation 27 (2008), 269-287.

\bibitem{DP} J. Dziuba\'nski, M. Preisner, {\it Spectral multiplier
theorem on Hardy spaces associated with semigroups generated by
Schr\"odinger operators with compactly supported potentials,}
unpublished manuscript.

\bibitem {DZ00} J. Dziuba\'nski, J. Zienkiewicz, {\it Hardy spaces associated with
 some Schr\"odinger operators,} Studia Math. 126
(1997),  149--160.

\bibitem{DZ1} J. Dziuba\'nski and J. Zienkiewicz,
 {\it Hardy space $H^1$ associated to Schr\"odinger operator
with potential satisfying reverse H\"older inequality}, Rev. Mat.
Iberoamericana, 15 (1999), no. 2, 279--296.

\bibitem{DZ0} J. Dziuba\'nski, J. Zienkiewicz, {\it  $H\sp p$ spaces for
Schr\"odinger operators,} Fourier analysis and related topics
(Bedlewo, 2000), 45--53, Banach Center Publ., 56, Polish Acad. Sci.,
Warsaw, 2002.


\bibitem{DZ2} J. Dziuba\'nski and J. Zienkiewicz, {\it $H\sp p$
spaces associated with Schr\"odinger operators with potentials from
reverse H\"older classes}, Colloq. Math. 98 (2003), no. 1, 5--38.

\bibitem {DZ5} J. Dziuba\'nski, J. Zienkiewicz, {\it Hardy spaces $H\sp 1$ for
Schr\"odinger operators with certain potentials,} Studia Math. 164
(2004), no. 1, 39--53.


\bibitem{DGMTZ} J. Dziuba\'nski, G. Garrig\'os, T. Mart\'inez, J.L. Torrea, J.
Zienkiewicz, {\it $BMO$ spaces related to Schr\"odinger operators
with potentials satisfying a reverse H\"older inequality,}  Math. Z.
249 (2005), no. 2, 329--356.

\bibitem{FeS} C. Fefferman and E.M. Stein, {\em $H^p$ spaces of
several variables}, Acta Math., 129 (1972), 137--195.

\bibitem{FS} G. Folland and E. Stein, {\em Hardy Spaces on
Homogeneous Groups,}  Princeton University Press, Princeton, NJ,
1982.

\bibitem{GC} J. Garcia-Cuerva, {\em Weighted $H^p$ spaces},
Dissertationes Math. (Rozprawy Mat.) 162 (1979).

\bibitem{GMST} J. Garcia-Cuerva, G. Mauceri, P. Sj\"ogren, and J. L.
Torrea, {\em Spectral multipliers for the Ornstein-Uhlenbeck
semigroup}, J. Anal. Math. 78 (1999), 281--305.


\bibitem{hebisch0} W. Hebisch, {\em A multiplier theorem for
Schr\"odinger operators}, Colloq. Math. 60/61 (1990), 659--664.

\bibitem{Hebisch} W. Hebisch, {\em Functional calculus for slowly
decaying kernels}, preprint, University of Wroc\l aw; also available
at {\tt http://www.math.uni.wroc.pl/$\tilde\ $hebisch}.

\bibitem{HZ} W. Hebisch, J. Zienkiewicz, {\it Multiplier theorems on
generalized Heisenberg groups II}, Colloq. Math. 69 (1995), 29-36.

\bibitem{Hof} S. Hofmann, G.Z. Lu, D. Mitrea, M. Mitrea, and L.X.Yan
{\em Hardy spaces associated with non-negative self-adjoint
operators satisfying Davies-Gafney estimates}, preprint 2008.


\bibitem{Horm} L. H\"ormander, {\em Estimates for translation invariant operators
in $L^p$ spaces}, Acta. Math., 104 (1960), 93-140.

\bibitem{L} R.H. Latter, {\em A decomposition of $H^p(\mathbb
R^n)$ in terms of atoms}, Studia Math., 62 (1977),  92--102.

\bibitem{MaS} R. Mac\'{\i}as and C. Segovia, {\em A decomposition into
atoms of distributions on spaces of homogeneous type}, Adv. in
Math., 33 (1979), 271-309.



\bibitem{MM} G. Mauceri and S. Meda, {\em Vector-valued multipliers
on stratified groups}, Rev. Mat. Iberoamericana 6 (1990), 141--154.

\bibitem{MM2} G. Mauceri, S. Meda, {\it BMO  and $H^1$  for the Ornstein-Uhlenbeck
operator}, J. Funct. Anal 252 (2007), 278--313.

\bibitem{MPR} G. Mauceri, M. Picardello, F. Ricci, {\it A Hardy
space associated with twisted convolution}, Adv. in Math. 39,
(1981), 270--288.

\bibitem{MS} D. M\"uller and E. M. Stein, {\em On spectral
multipliers on Heisenberg and related groups,} J. Math. Pures Appl.
73 (1994), 413--440.


\bibitem{S} {A. Sikora}, {\em Multilpicateurs associ\'es aux souslaplaciens
sur les groupes homog\`enes}, C.R. Acad. Sci. Paris 315 (1992),
417--419.

\bibitem{S2} A. Sikora, {\em Riesz transform, Gaussian bounds and the method
of wave equation,} Math. Z.,  247 (2004), no. 3,
 643--662.

\bibitem{Stein} E. Stein, {\em Topics in harmonic analysis related to the Littlewood-Paley theory}, Princeton
University Press, Princeton, NJ, 1970.

\bibitem{SW} E. Stein, G. Weiss, {\it On the theory of harmonic
functions of several variables. I. The theory of $H^p$-spaces,}
Acta Math., 103 (1960), 25-62.

\bibitem{Zien} J. Zienkiewicz, Personal communication.





\end{thebibliography}
\end{document}